\documentclass[11pt, a4paper]{amsart}

\usepackage{amsfonts}
\usepackage{amssymb}
\usepackage{amsmath, amsthm,color}
\usepackage{hyperref}
\usepackage{pdfsync}
\usepackage{bbm, dsfont}
\usepackage{graphicx}
\usepackage[margin=0.96in]{geometry}

\def\sign{\operatorname{sgn}}

\newcommand{\pd}{\partial}
\def\dist{\operatorname{dist}}

\newcommand{\eps}{\varepsilon}
\newcommand{\la}{\lambda}
\newcommand{\vf}{\varphi}

\newcommand{\bR}{\mathbb{R}}

\newcommand{\bE}{\mathbb{E}}
\newcommand{\bP}{\mathbb{P}}

\newcommand{\cD}{\mathcal{D}}

\newcommand{\HH}{\{-1,1\}^n}

\newtheorem{theorem}{Theorem}[section]

\newtheorem{remark}[theorem]{Remark}

\numberwithin{equation}{section}

\begin{document}
\title[Poincar\'e inequality on Hamming cube]{Improving constant in end-point Poincar\'e inequality on Hamming cube}
\author[Paata Ivanisvili, Dong Li, Ramon van Handel, Alexander Volberg]{Paata Ivanisvili, Dong Li, Ramon van Handel, Alexander Volberg}
\thanks{ Volberg is partially supported by the NSF DMS-1600065.  }
\address{Department of Mathematics, Princeton University}
\email{rvan@princeton.edu \textrm{(Ramon Van Handel)}}
\address{Department of Mathematics, Princeton University}
\email{paata.ivanisvili@princeton.edu \textrm{(Paata ivanisvili)}}
\address{Department of Mathematics, Hong Kong University of Science and Technology}
\email{madli@ust.hk \textrm{(Dong Li)}}


\address{Department of Mathematics, Michigan State University, East Lansing, MI 48823, USA}
\email{volberg@math.msu.edu \textrm{(Alexander\ Volberg)}}
\makeatletter
\@namedef{subjclassname@2010}{
  \textup{2010} Mathematics Subject Classification}
\makeatother
\begin{abstract} 
We improve the constant $\frac{\pi}{2}$ in $L^1$-Poincar\'e inequality on Hamming cube. For Gaussian space the sharp constant in $L^1$ inequality is known, and it is  $\sqrt{\frac{\pi}{2}}$. For Hamming cube the sharp constant is not known, and $\sqrt{\frac{\pi}{2}}$ gives an estimate from below for this sharp constant. On the other  hand, L. Ben Efraim and F. Lust-Piquard  have shown an estimate from above: $C_1\le \frac{\pi}{2}$. There are at least two other independent proofs of the same estimate from above  (we write   down  them in this note). Since those proofs are very different from the proof of  Ben Efraim and Lust-Piquard  but gave the same constant, that might have indicated that constant is sharp. But here we give a better estimate from above, showing that $C_1$ is strictly smaller than  $\frac{\pi}{2}$.
It is still not clear whether $C_1> \sqrt{\frac{\pi}{2}}$. We discuss this circle of questions.
\end{abstract}
\maketitle 

\section{Introduction}
\label{intro}

In \cite{BELP} the following inequality is proved
\begin{equation}
\label{L1P1}
\| f-\bE f\|_1 \le \frac{\pi}{2} \|\nabla f\|_1,
\end{equation}
where $f$ is a function on Hamming cube $\HH$. Pisier proved that in Gaussian space the following Poincar\'e inequality holds with the sharp constant $\sqrt{\pi/2}$:
\begin{equation}
\label{G1P1}
\| f-\bE f\|_1 \le \sqrt{\frac{\pi}{2} }\|\nabla f\|_1,
\end{equation}
If we denote by $C_1$ the best constant in the Poincar\'e inequality in $L^1(\HH)$, then we see that
$$
\sqrt{\frac{\pi}{2} } \le C_1 \le  \frac{\pi}{2}\,.
$$

In this note we improve the right estimate.  Let $\Delta$ be Laplacian on $\HH$ (negative operator).  Let $P_t:= e^{t\Delta}$ be a corresponding semi-group.

The proof of \eqref{L1P1} in \cite{BELP} is striking. To obtain this estimate the authors adapt Pisier's proof to the Hamming cube.  For this  they  lift 
the problem about functions to non-commutative problem about matrices.  
After that they manage to represent operator $P_t$ as a compression of a semi-group acting on a non-commutative space of $2^n \times 2^n$ matrices, and then the rest of the argument relies on the non-commutative Khinthchin inequality.  
This lifting of a problem about usual functions to a non-commutative setting is immensely beautiful and enticing, but also a bit mysterious.

There are ``commutative" proofs of the estimate from above in $L^1(\HH)$-Poincar\'e inequality. 
We present them in Section \ref{BellmanC}. They give exactly the same constant $\pi/2$ (or worse) as in 
\cite{BELP} but they use a sort of ``Bellman function" monotonicity idea. We learnt them from the book of Bakry--Gentil--Ledoux \cite{BGL}, Chapter 8.

This persistence of $\pi/2$ constant in three very different proofs could have been suggestive. But here we prove that sharp constant is smaller than $\pi/2$. What it is remains enigmatic.

\section{Dual problem}
\label{dual}

We write a dual problem as follows. Let $f \in L^1(\HH)$ and $g\in L^\infty (\HH)$. Let $\bE f=0$. Then
$$
f-\bE f = -\int_0^\infty \frac{d}{dt}P_t f \, dt\,.
$$
and hence
\begin{align*}
&(f - \bE f, g) =  -\int_0^\infty (\Delta P_t f , g)\, dt = -\int_0^\infty (\Delta  f , P_t  g)\, dt =
\\
&  - ( \Delta f , \int_0^\infty P_t  g \, dt) =   - ( \nabla f ,  \int_0^\infty  \nabla P_t  g\, dt) \,.
\end{align*}

Therefore,
$$
\bE |f-\bE f| \le \|\nabla f\|_1 \cdot \sup_{\|g\|_\infty \le 1} \Big\| \int_0^\infty  \nabla P_t  g\, dt\Big\|_\infty\le  \|\nabla f\|_1 \cdot \sup_{\|g\|_\infty \le 1} \int_0^\infty  \|\nabla P_t  g\|_\infty\, dt  \,.
$$

So we will  estimate
$$
  \sup_{\|g\|_\infty \le 1} \int_0^\infty  \|\nabla P_t  g\|_\infty\, dt 
 $$
 by $C_{dual}\|g\|_\infty$.   We will prove that $C_{dual}<\frac{\pi}{2}$. We just showed that $C_1\le C_{dual}$. There is a very good possibility that $C_1< C_{dual}$. We discuss that in this note, where we talk about the $Curl$ space--see below in Section \ref{curl}.
 
\section{Integral operator}
\label{intO}

Let us consider $g\mapsto \nabla P_t  g$ as an integral operator and let us write down its kernel.  Consider independent random variables $\{y_1,\dots, y_n\}$, which are
$\rho$-correlated with standard Bernoulli independent random variables $\{x_1,\dots, x_n\}$. If $y_i= x_i$ with probability $\frac{1+\rho}{2}$ and $y_i= - x_i$ 
with probability $\frac{1-\rho}{2}$, then we have them exactly $\rho$-correlated  $\bE y_i x_i = \rho$.  Given a fixed string $x\in \HH$, we can write
$$
\bE_{{\bf y} \sim N_\rho(x)}{\bf y}^S = \rho^{|S|} x^S,
$$
where $S$ is a multi-index of $0$ and $1$, $x^S$ is a corresponding polynomial, and ${\bf y} \sim N_{\rho}(x)$ means distribution $\rho$-correlated  independent random variables.

Putting  $\rho= e^{-t}$ we get
$$
\bE_{{\bf y} \sim N_\rho(x)} g({\bf y}) = (P_t g)(x)\,.
$$
Since 
$$
\nabla P_t  g = e^{-t} P_t \nabla g
$$
we can apply this to $\pd_1 g$.
$$
(P_t \pd_1g)(x) = \bE_{{\bf y} \sim N_\rho(x)} \pd_1 g({\bf y})\,.
$$
Now we want to find $\vf_1({\bf y})$ such that
$$
\bE_{{\bf y} \sim N_\rho(x)} \pd_1 g({\bf y}) =\bE_{{\bf y} \sim N_\rho(x)} \vf_1({\bf y}) g({\bf y})\,.
$$
But $\pd_1$ eliminates all polynomials that do not have ${\bf y}_1$ and cross off ${\bf y}_1$ from other polynomials. So $\vf_1$  such that
$$
\bE_{{\bf y} \sim N_\rho(x)}\vf_1({\bf y}) {\bf y}_k =\delta_{1k}\,.
$$
Clearly the following $\vf_1$ works:
$$
\vf_1(y) = \frac{y_1-  \bE_{{\bf y} \sim N_\rho(x)} {\bf y}_1}{Var[{\bf y_1}]}= \frac{y_1-  \rho x_{1}}{1-\rho^2} =  \frac{y_1-  e^{-t}x_{1}}{1-e^{-2t}}\,.
$$
Combining all that we get the integral representation of $(P_t \nabla g)(x)$:
\begin{equation}
\label{intKernel}
(P_t \nabla g)(x) = \bE_{{\bf y} \sim N_\rho(x)} \Big( (\frac{{\bf y_1}-  e^{-t}x_{1}}{1-e^{-2t}}, \dots, \frac{{\bf y_n}-  e^{-t}x_{n}}{1-e^{-2t}}) g({\bf y})\Big)\,.
\end{equation}
From \eqref{intKernel} we get ($\la=\la_1,\dots, \la_n)$ (a unit vector in $\bR^n$)
\begin{equation}
\label{3norms}
\|(P_t \nabla g)\|_{\infty}\le \frac1{\sqrt{1-e^{-2t}}} \|g\|_\infty \left\| \sup_{\la: \|\la\|_2=1} \bE_{{\bf y} \sim N_\rho(x)}\Big |\sum_{j=1}^n\la_j \frac{ {\bf y}_j-  e^{-t}x_{j}}{\sqrt{1-e^{-2t}}}\Big|\right\|_{\infty}
\end{equation}
To estimate the right hand side without loss of generality we can assume that $x_{j}=1$ for all $j=1,\ldots, n$. Indeed, this follows from the fact that we are taking supremum over all $\lambda \in \mathbb{S}^{n-1}$, and since ${\bf y}-e^{-t}x_{j}$ takes values $x_{j} (1-e^{-t})$ and $-x_{j}(1+e^{-t})$ we can absorb the signs $x_{j}$ into the values of $\lambda_{j}$. 

The random variables $\frac{ {\bf y}_j-  e^{-t}}{\sqrt{1-e^{-2t}}}$ are independent, and they have the same distribution as random variables
$\xi_i^t$ that assume value  $\sqrt\frac{1- e^{-t}}{1+ e^{-t}}$ with probability $\frac{1+ e^{-t}}{2}$ and  $-\sqrt\frac{1+ e^{-t}}{1- e^{-t}}$ with probability $\frac{1- e^{-t}}{2}$. 
 
 Hence, for $\|g\|_\infty\le 1$, we have
 \begin{equation}
\label{est1}
\|\nabla(P_t  g)\|_\infty= \frac{e^{-t}}{\sqrt{1-e^{-2t}}}  \sup_{\la: \|\la\|_2=1}\big \|\sum \la_j \xi_j^t\big \|_{L^1(\bP)}\,.
\end{equation}

\bigskip

Notice that $\int \xi_j^t \, d\bP =0$, $\int (\xi_j^t)^2\, d\bP =1$. Thus, there is a trivial estimate
\begin{equation}
\label{est2}
\|\nabla(P_t  g)\|_\infty= \frac{e^{-t}}{\sqrt{1-e^{-2t}}}  \sup_{\la: \|\la\|_2=1}\big \|\sum \la_j \xi_j^t\big \|_{L^2(\bP)} \le \frac{e^{-t}}{\sqrt{1-e^{-2t}}}\,.
\end{equation}
 The first estimate here is just \eqref{est1}, the second one is just a trivial fact that $\|\sum \la_j \xi_j^t\big \|_{L^2(\bP)}^2 = \la_1^2+\dots +\la_n^2=1$.
 
 Using $\int_0^\infty  \frac{e^{-t}}{\sqrt{1-e^{-2t}}} dt =\int_0^1\frac{du}{\sqrt{1-u^2}}=\frac{\pi}{2}$ we see \eqref{L1P1} one more time.
 
 \bigskip
 
 To improve this estimate it is sufficient to prove the following theorem.
 \begin{theorem}
 \label{34}
 Let $p\in [\frac12,1\big]$.
 Consider independent random variables $\eps_j$, $j=1, \dots, n$, having values $\sqrt{\frac{1-p}{p}}$ with probability $p$ and $-\sqrt{\frac{p}{1-p}}$ with probability $1-p$.  Then $\sup_{\la\in \bR^n: \|\la\|_2=1} \|\sum\la_j \eps_j\|_1 =: q(p) <1$ for $p$ lying  in a small interval around $3/4$.
 \end{theorem}
 
\medskip
 
 \begin{remark}
 In fact, the proof will show that $q(p) <1$ for all $p\in (\frac12, 1)$.
 \end{remark}
 
 \medskip
 
  In the next two section we prove this theorem.

 \medskip
 
 \noindent{\bf Acknowledgement.} We are grateful to Sergei Bobkov who indicated to us the article \cite{BC}. The proof there, even though it is different from the proof below, encouraged us.
 
\section{Maximum $\la$ is separated from $1$}
\label{mala}

Everything is real-valued below. We will need the $8$-th moment in the calculation below.

Let $\xi_i$, $i=1, \dots, n$ be our random Bernoulli variables with $\bE \xi=0, \bE\xi^2=1$, $1-p, p$ probability of values $\sqrt{\frac{p}{1-p}}, -\sqrt{\frac{1-p}{p}}$, and independent. Let $\la=(\la_1,\dots, \la_n)$ be a point on the unit sphere. In this section we consider the case
\begin{equation}
\label{099}
\max |\la_k|^2 \le 0.99\,.
\end{equation}

We want to prove that independently of $n$ under this assumption above and with a certain $\eps>0$, which will depend only on a constant  $p$ chosen later in \eqref{cheps}, we will have
\begin{equation}
\label{09999}
\max |\la_k|^2 \le 0.99\Rightarrow\bE |\sum\la_i \xi_i| \le (1-\eps^2)\big(\bE (\sum\la_i \xi_i)^2\big)^{1/2}=1-\eps^2\,.
\end{equation}

Denote $Y:=  |\sum\la_i \xi_i| $ and notice that if the opposite happens for some $\la$ satisfying \eqref{099}, then

$$
\bE (Y-\bE Y)^2 =1 - \big(\bE Y\big)^2 < 1-(1-\eps^2)^2 \le 2\eps^2\,.
$$
So if \eqref{09999} does not hold, then on a large probability $Y$ is close to $\bE Y$ (and $\bE Y$ is close to  $1$ of course) for certain $\la$ satisfying \eqref{099}.
Hence, for this $\la$,

\begin{equation*}
\text{with probability}\quad  1-2\eps \quad \text{one has} \,\, \big | |\sum\la_i \xi_i| -1\big|^2 \le \eps+2\eps^2\,.
\end{equation*}
We took here into account that if  the opposite to \eqref{09999}  happens, then $0\le 1-\bE Y \le 1-\sqrt{1-2\eps^2}$. 
In particular, we obtain 
\begin{align}\label{cheb}
P\left( \left| |\sum \lambda_{j} \xi_{j}|^{2}-1\right| \leq \sqrt{\varepsilon+2\varepsilon^{2}}(1+\sqrt{\varepsilon+2\varepsilon^{2}})\right)\geq 1-2\varepsilon
\end{align}

\bigskip

Now let us bring \eqref{cheb} (with a certain $\eps$ chosen below with the help of constant $B$ from \eqref{reas} below)  to a contradiction if \eqref{099} holds.

\medskip

Consider $\ell=\sum\la_i\xi_i$, then
$$
\ell^2 = \sum \la_i^2\xi_i^2 + 2\sum_{i<j} \la_i\la_j \xi_i\xi_j\,.
$$
So
\begin{align}
\label{ell4}
&\ell^4 = \sum \la_i^4 \xi_i^4 + 2\sum_{i<j} \la_i^2 \la_j^2 \xi_i^2 \xi_j^2 + 4 \sum_{i<j} \la_i^2\la_j^2 \xi_i^2\xi_j^2 + \notag
\\
&4\sum_{i<j, k<m, (i, j)\neq (k,m)} \la_i\la_j \xi_i\xi_j  \la_k\la_m \xi_k\xi_m\, + 4 \sum_{m,\, i<j}\lambda_{m}^{2}\xi_{m}^{2}\lambda_{i}\lambda_{j}\xi_{i}\xi_{j}
\end{align}
Of course $\bE \ell^2 =1$. Now calculate $Var[\ell^2]$. By \eqref{ell4} we have
\begin{align*}
&Var[\ell^2] = \bE(\ell^2-1)^2 = \bE \ell^4 -2\bE \ell^2 +1 = \bE \ell^4 -1=
\\
& \bE\xi^4 \sum\la_i^4 + 6 \sum_{i<j} \la_i^2 \la_j^2 -1 \geq \sum \lambda_{j}^{4} +3(1-\sum \lambda_{j}^{4}) -1=\\
&2(1-\sum \lambda_{j}^{4})\geq 2\cdot 0.01=0.02 
\end{align*}
Here we have used $\mathbb{E} \xi^{4}\geq 1$ and $\lambda_{j}^{4}\leq 0.99\lambda_{j}^{2}$ for all $j=1,\ldots, n$. 

\medskip

Let $X:= |\ell^2-1|$. If 
\begin{equation}
\label{reas}
\bE X^4 \le B (\bE X^2)^2,
\end{equation}
then the Paley--Zygmund estimate applied to $X^2$ says 
$$
\bP ( X \ge t (\bE X^2)^{1/2}) \ge (1-t^2)^2 \frac1{B}, \quad t\in (0,1)\,.
$$
Let us take \eqref{reas} for granted and let us then  see what Paley--Zygmund estimate gives us with $t=1/2$.
Since we estimated $\bE X^2$ from below as follows
$$
\bE X^2 \ge 0.02
$$
we get
\begin{align}
\label{cheb01}
\bP (  |\ell^2-1| \ge 0.07) \ge \bP ( X \ge \frac12\frac{\sqrt{2}}{10}) \ge \frac9{16} \frac1{B}\,.
\end{align}
This contradicts \eqref{cheb}. Indeed, summing up (\ref{cheb}) and (\ref{cheb01}) we obtain 
\begin{equation}
\label{cheps}
\bP(|\ell^{2}-1|\geq 0.07) + \bP\big(|\ell^{2}-1|\leq  \sqrt{\varepsilon+2\varepsilon^{2}}(1+\sqrt{\varepsilon+2\varepsilon^{2}})\big) \geq \frac{9}{16B}+1-2\varepsilon\,.
\end{equation}
Now it remains to take $\varepsilon$ sufficiently small to get a contradiction.



So if we  prove \eqref{reas} with $B$ depending on $p$ but independent of $n$, we would prove that we have the drop in norm as in \eqref{09999} with $\eps$ sufficiently small depending only on $p$.

\bigskip

To see \eqref{reas} with $B$ independent of $n$ and independent of $\la_i$ satisfying \eqref{099}, $\sum\la_i^2=1$, let us recall that $X= |\ell^2-1|$ and we already estimated
$\bE X^2 = \bE (\ell^2-1)^2 \ge 0.02$ from below, and this estimate depends only on assumption \eqref{099}.

To prove \eqref{reas} we now just need to estimate $\bE X^4 = \bE (\ell^2-1)^4$ from above by $C(p)<\infty$. For this we need only to estimate 
$\bE\, \ell^8$. Looking at \eqref{ell4}, we square it and integrate. It is clear then
that only sums involving $\la_k^{2m}$, $m=1,2, 3, 4$,  will survive. It is now easy to see that $\bE \,\ell^8\le C(p)$, where $C(p)$ is bounded if we do not make $p\to 1$.  This just because $\sum_{j=1}^n\la_j^{2m}\le 1,\, m=1,2,\dots$, and because of the obvious estimate $\bE \, |\xi_i|^{2m} \le K(p, m, \delta)<\infty$ if $p\in [\frac12, 1-\delta\big)$.

Hence, we  have
\begin{equation}
\label{reas1}
\bE X^4 \le B(p) (\bE X^2)^2, \quad B(p) \le B_\delta<\infty, \,\, p\in [1/2, 1-\delta\big)\,.
\end{equation}

So for $p=\frac34$ and also for $p$ in a small fixed (independent of $n$)  neighborhood of $p=\frac34$ we have a definite drop in norm. In other words, we have the H\"older inequality with constant strictly smaller than $1$ independently of $n$ and  independent of $\la$ satisfying \eqref{099}.

\section{Maximum is close to $1$}
\label{close}

What if the $\max_k |\lambda_k|^2=:\la_1^2$ is in $[0.99, 1]$?
Here we should think that $p\in (\frac12 + \delta, 1\big]$. For example $p$ is in a small fixed interval around $\frac34$.

Then we write for $p=\frac34$:
\begin{align*}
&\bE |\sum\la_i \xi_i| \le \la_1 \bE |\xi_1| +\bE |\sum_{i=2}^n\la_i \xi_i|  \le  \la_1 \bE |\xi_1| + \big(\bE |\sum_{i=2}^n\la_i \xi_i|^2\big)^{1/2}
\\
& \le 2\sqrt{(1-p)p} +\sqrt{1-\la_1^2} \le \frac{\sqrt{3}}{2} +\frac1{10} < 0.87+0.1= 0.97.
\end{align*}
Again we have a fixed drop in H\"older inequality independent on $n$. And the same drop happens trivially in a small neighborhood of $p=\frac34$, and this neighborhood does not depend neither on $n$ nor on $\la$ such that $\max_k |\lambda_k|^2\in [0.99, 1]$

\bigskip

\section{Bellman proof of Maurey--Pisier estimate on gaussian space}
\label{BellmanG}

We want to explain two proofs of $L^1$-Poincar\'e inequality on Hamming cube that can be derived from the literature.

We already mentioned that there are  other proofs of the estimate $\bE |f -\bE f|$ via $ C\,\bE |\nabla f|$ on Hamming cube. These can be called ``Bellman function proofs", they also gave $C=\frac{\pi}{2}$. Let us briefly recall one of them. First we recall how to use ``Bellman function approach" to derive the sharp constant  $\sqrt{\frac{\pi}{2}}$ in gaussian space.

The proof from \cite{BGL} below is longer than a very short proof of Maurey--Pisier, 
but it has the advantage that it can be somewhat generalized to Hamming cube $L^1$-Poincar\;e inequality. 

Let $ \Phi(x)$ be the gaussian error function,
$$
\Phi(x) =\frac1{\sqrt{2\pi}} \int_{-\infty}^x e^{-y^2/2}\, dy\,.
$$
Let us consider the ``gaussian isoperimetric profile":
$$
I= \Phi'\circ \Phi^{-1}: [0,1]\to \Big[0, \frac1{\sqrt{2\pi}}\Big]\,.
$$

Let us first prove Maurey--Pisier estimate by ``Bellman function" approach borrowed from \cite{BGL}, Chapter 8.
Another, and more elegant proof, can be found in \cite{Pi}. But it is very ``gaussian" and difficult to invent a simple way to adapt it to the Hamming cube situation.
In a certain sense paper \cite{BELP} does such an adaptation but in a very fascinating non-obvious way.

Let $P_t= e^{t\Delta}$ denote the Ornstein--Uhlenbeck semigroup on $\bR^n$, $\Delta$ is the Ornstein--Uhlenbeck Laplacian. Function $I^2$ will play the part of ``Bellman function" in the sense that a certain monotonicity involving the semigroup $P_t$ and function $I^2$ will be crucial.  We first consider only $f$ such that $0 \le f\le 1$. Obviously, 
$$
\big[I(P_t f)\big]^2 - \big[P_t(I(f))\big]^2 = -\int_0^t \frac{d}{ds} \big[P_s(I(P_{t-s} f))\big]^2\, ds
$$
Combine this with
\begin{align*}
&-\frac{d}{ds} \big[P_s(I(P_{t-s} f))\big]^2= - 2 P_s(I(P_{t-s} f)) \cdot P_s \bigg( \Delta\big( I(P_{t-s} f)\big)- I'(P_{t-s} f)\cdot \Delta P_{t-s} f\bigg)=
\\
&  - 2 P_s(I(P_{t-s} f)) \cdot P_s\bigg( I''(P_{t-s} f) \cdot |\nabla P_{t-s} f|^2\bigg) = 2 P_s(I(P_{t-s} f))\cdot P_s\bigg(\frac{ |\nabla P_{t-s} f|^2}{I(P_{t-s}f)}\bigg) \ge
\\ 
&2P_s(I(P_{t-s} f))\cdot \frac{\big[ P_s |\nabla P_{t-s} f|\big]^2}{ P_s (I(P_{t-s} f))}=\big[ P_s |\nabla P_{t-s} f|\big]^2 \,.
\end{align*}

The third equality here is because $I''= -\frac1{I}$.  The inequality is just Cauchy-Schwartz inequality: $\int\frac{A^2}{B} d\mu \ge \frac{[\int A\, d\mu]^2}{\int B\, d\mu}$. The second equality is the chain rule in this form: for any smooth $G$ and test function $g$ on $\bR^n$
\begin{equation}
\label{chainR}
\Delta G(g) = G'(g) \Delta g + G''(g) |\nabla g|^2\,.
\end{equation}
We warn the reader that only this last simple equality will fail on the cube. 

Let us combine the estimate 
\begin{equation}
\label{kneeOU}
\big[I(P_t f)\big]^2 - \big[P_t(I(f))\big]^2 \ge 2\int_0^t \big[ P_s |\nabla P_{t-s} f|\big]^2 \, ds,
\end{equation}
 which has been just obtained, with the following well-known (and easy, see, e. g., \cite{BGL}) estimate for the Ornstein-Uhlenbeck semigroup:
\begin{equation}
\label{gradOU}
|\nabla P_s g| \le e^{-s} P_s|\nabla g|\,.
\end{equation}
Then we get
$$
\big[I(P_t f)\big]^2 - \big[P_t(I(f))\big]^2 \ge 2 \int_0^t  e^{2s} |\nabla P_t f|^2\, ds \,.
$$
Thus
$$
0 \le f\le 1\Rightarrow |\nabla P_t f|^2 \le \frac{1}{e^{2t} -1} \big(\big[I(P_t f)\big]^2 - \big[P_t(I(f))\big]^2\big)\le  \frac{\big[I(P_t f)\big]^2}{e^{2t}-1}\,,
$$
and so
$$
0 \le f\le 1\Rightarrow |\nabla P_t f| \le  \frac{I(P_t f)}{\sqrt{e^{2t}-1}} \le \frac{1}{\sqrt{2\pi}} \frac{1}{\sqrt{e^{2t}-1}}
\,,
$$
Hence, 
\begin{equation}
\label{positiveG}
0 \le f\le 1\Rightarrow\int_0^\infty |\nabla e^{t\Delta} f| dt \le  \frac{1}{\sqrt{2\pi}}\int_0^\infty  \frac{1}{\sqrt{e^{2t}-1}}\, dt = \frac{1}{\sqrt{2\pi}} \frac{\pi}{2}=\frac12 \sqrt{\frac{\pi}{2}}\,.
\end{equation}
Finally, this immediately implies
\begin{equation}
\label{boundedG}
\int_0^\infty |\nabla e^{t\Delta} f| dt \le   \sqrt{\frac{\pi}{2}}\|f\|_\infty\,.
\end{equation}

This gives the sharp constant in $L^1$-Poincar\'e inequality on gaussian space:
\begin{equation}
\label{PoiG}
\bE_g |f-\bE f| \le \sqrt{\frac{\pi}{2}} \bE_g |\nabla f|\,.
\end{equation}
This proof can be somewhat generalized to Hamming cube $L^1$-Poincar\;e inequality. 
Since the simple chain rule \eqref{chainR} will not work, the proof should be modified and the constant jumps: 
strangely enough, it becomes $\frac{\pi}{2}$. Here is the reasoning.

It would be nice to have on Hamming cube $C^n$
the variant of our usual relationship \eqref{chainR}, e. g.,  to have it in this form:
$$
I'(g) \Delta g - \Delta [I(g)] \ge c [-I''(g)] |\nabla g|^2
$$
with some constant $c$, $c\le 1$. This is how we wish to replace \eqref{chainR}, which is false on Hamming cube.
On gaussian space this is {\it equality} with $c=1$ as we saw in \eqref{chainR} with $G=I$

\section{Bellman proofs of Ben Efraim--Lust-Piquard estimate on Hamming cube}
\label{BellmanC}

On cube this becomes two point inequality
$$
-x_j \pd_j g \cdot I'(g) + x_j \pd_j (I(g)) \ge c [-I''(g)] |\pd_j g|^2,
$$
where $\pd_j g = (g(x_j=1)- g(x_j=-1))/2$.
Or, denoting $g(x_j=1) =: b, g(x_j=-1)=: a$:
$$
-(b-a) I'(b) + (I(b)-I(a))\ge \frac{c}{2} [-I''(b)] (a-b)^2\ge 0\,.
$$
This is
\begin{equation}
\label{2point1}
I(b)-I(a) - I'(b) (b-a) + \frac{c}{2} I''(b) (a-b)^2 \ge 0 
\end{equation}
that suppose to be valid for all pairs $a, b$ in $[0,1]$. Fix $a$ and tend $b$ to one of the end points $0$ or $1$.
Let, for example, $b\to 0$. Notice that $I''(b) \to -\infty$ as $-\frac{1}b\frac{1}{\sqrt{\log 1/b}}$, and notice that $I'(b)\to +\infty$  as $\sqrt{\log 1/b}$. Hence \eqref{2point1} never can be true for $b$ allowed to tend to end points.

\bigskip

We saw that \eqref{2point1} cannot hold for all pairs $a, b\in [0,1]$.

\medskip

So our first try to circumvent the lack of the chain rule is not successful. 

But we can ask another question:
what is the largest possible constant $k>0$  such that

\begin{align}
\label{2point_k}
I(b)-I(a) - I'(b) (b-a) -k\,(a-b)^2/2 \ge 0\quad \forall a, b\in [0,1]?
\end{align}

The answer of course is obvious, $k=\sqrt{2\pi}$. 
Indeed,
\begin{align}
\label{2point3}
&I(b)-I(a) - I'(b) (b-a) + \frac12 \max_{c\in [0,1]} [I''(c)] (a-b)^2\ge 0,\quad \forall a, b\in [0,1],
\\
& \text{that is}\quad I(b)-I(a) - I'(b) (b-a) -\sqrt{2\pi}   (a-b)^2/2  \ge 0\,. \notag
\end{align}

\bigskip

To check that constant $k$  in \eqref{2point_k} cannot be bigger than $\sqrt{2\pi}$ just make $a$ and $b$ go to $1/2$. 
The previous estimate \eqref{2point3} gives us
\begin{equation}
\label{min}
-x_j \pd_j g \cdot I'(g) + x_j \pd_j (I(g)) \ge \min [-I''(g)] |\pd_j g|^2 = \sqrt{2\pi}|\pd_j g|^2,
\end{equation}
Or,
\begin{equation}
\label{chainC}
I'(g) \Delta g - \Delta [I(g)] \ge \sqrt{2\pi} |\nabla g|^2
\end{equation}

Let now $\Delta$ be the Laplacian on Hamming cube, and $P_t= e^{t\Delta}$ be the corresponding flow on cube.
Then we get the analog of \eqref{kneeOU} (but without squares over $[\cdot]$):
\begin{equation}
\label{kneeC}
\big[I(P_t g)\big] - \big[P_t(I(g))\big] \ge \sqrt{2\pi}\int_0^t \big[ P_s |\nabla P_{t-s} g|^2\big] \, ds\ge  \sqrt{2\pi}\int_0^t \big[ P_s |\nabla P_{t-s} g|\big]^2\ge  \sqrt{2\pi} |\nabla P_{t} g|^2\int_0^t  e^{2s} ds ,
\end{equation}

\begin{equation}
\label{Ramon2}
\frac12 (e^{2t}-1)\sqrt{2\pi} |\nabla e^{t\Delta} g|^2 \le I(e^{t\Delta} g) - e^{t\Delta} I(g) \le I(e^{t\Delta} g) \le \frac1{\sqrt{2\pi}}\
\end{equation}
Or,
$$
\frac12 (e^{2t}-1) \sqrt{2\pi} |\nabla e^{t\Delta} g|^2  \le \frac1{\sqrt{2\pi}}\
$$
Hence, for 
$0\le g\le 1$ we have
\begin{equation}
\label{Ramon}
\ |\nabla e^{t\Delta} g\|_\infty \le \frac1{\sqrt{2\pi}} \frac{\sqrt{2}}{\sqrt{e^{2t}-1}}
\end{equation}
Hence for any bounded positive $g$
$$
\|\nabla e^{t\Delta} g\|_\infty \le \frac1{\sqrt{2\pi}} \frac{\sqrt{2}}{\sqrt{e^{2t}-1}}\|g\|_\infty\,.
$$
Now
$$
\int_0^\infty  \frac{1}{\sqrt{e^{2t}-1}} dt =\frac{\pi}{2}\,.
$$
So for positive $g$
\begin{equation}
\label{Ramon0}
\int_0^\infty \|\nabla e^{t\Delta} g\|_\infty dt \le \frac{\pi}{2} \cdot \frac1{\sqrt{\pi}}\|g\|_\infty\,.
\end{equation}
So for all $g$,
\begin{equation}
\label{Ramon1}
\int_0^\infty \|\nabla e^{t\Delta} g\|_\infty dt \le \sqrt{\pi} \|g\|_\infty\,.
\end{equation}

This is worse than $\frac{\pi}{2}  \|g\|_\infty$. Looking at the above proof, we immediately see that $I= \Phi'\circ \Phi^{-1}$ being optimal for the estimate in gaussian space might be not optimal on cube. In fact, replacing $\Delta( I(P_{t-s} f)\big)- I'(P_{t-s} f)\cdot \Delta P_{t-s} f$  by  $\Delta(B(P_{t-s} f)\big)- B'(P_{t-s} f)\cdot \Delta P_{t-s} f$ we can see that we should find $B: [0,1]\to \bR_+$ such that
\begin{equation}
\label{min}
M_B:=\frac{\max_{[0,1]} B(x)}{\min_{[0,1]} [-B''(x)]}\to \min\,.
\end{equation}
If we call this minimum $M$, we obtain (just by repeating the reasoning of Section \ref{BellmanG}) the following estimate:
$$
0\le f\le 1\Rightarrow |\nabla P_t f| \le \sqrt{2}\sqrt{M} \frac{1}{\sqrt{e^{2t}-1}},
$$
For $B=I$ we have $M_I=\frac{1}{2\pi}$. And this implies \eqref{Ramon0}, and \eqref{Ramon1}.

But the choice $B(x)= x(1-x)$ gives minimum in \eqref{min}, and it is $\frac18$. Then
$$
0\le f\le 1\Rightarrow \int_0^\infty  |\nabla P_t f|  dt \le \sqrt{2}\frac1{2\sqrt{2}} \frac{\pi}{2} =\frac{\pi}4\,.
$$
Hence, 
$$
\int_0^\infty  |\nabla P_t f|  dt \le  \frac{\pi}{2}\|f\|_\infty\,.
$$
And this way we get a commutative proof of Ben Efraim--Lust-Piquard estimate:
$$
\bE |g -\bE g| \le \frac{\pi}{2} \bE |\nabla g|\,.
$$

\section{Discussion}
\label{discu}

\subsection{Functions with only two values have constant $\sqrt{\frac{\pi}{2}}$ in $L^1$-Poincar\'e inequality}
\label{two values}

Incidentally, the question of validity of the Gaussian inequality constant $\sqrt{\frac{\pi}{2}}$ on the 
hypercube is particularly mysterious, for the following reason. In the 
Gaussian case, the extremizer that attains the optimal constant is the 
indicator function of a halfspace of probability $1/2$. In particular, it is 
a fortiori a function of the form $f=1_A$. But if we restrict the hypercube 
$L^1$-Poincar\'e inequality only to indicators $f=1_A$, then the inequality 
does hold with the same constant as in the Gaussian case and this is 
optimal. 

This follows from Bobkov's inequality on the cube. In fact, it is known that
\begin{equation}
\label{I(p)}
\max_{p\in [0,1]} \frac{2p(1-p)}{I(p)}=\sqrt{\frac{\pi}{2}} \,.
\end{equation}

Now notice that any function $f$ having only two values can be made to a function having values $0,1$ by linear transformation
$f\to af +b$.  And this transformation does not change the constant in Poincar\'e inequality.Then we can think that $1$ is assumed with probability $p\in (0,1)$. Hence $\bE f=p$, $\bE|f-\bE f| = 2p(1-p)$.

Let $B(x, y):= \sqrt{I^2(x) +y^2}$. Here $I(x)= \Phi'\circ \Phi^{-1}(x)$, where $\Phi$ is the Gaussian error function. It is called Bobkov's function and Bobkov \cite{Bob} proved that for any $f:\{-1,1\}^n \to [0,1]$ the following inequality holds:
\begin{equation}
\label{Bobk1}
I(\bE f)= B(\bE f, 0) \le \bE \big[B(f, |\nabla f|)\big]\,.
\end{equation}

For functions as above having values $0,1$ only, this becomes
\begin{equation}
\label{Bobk1}
I(\bE f) \le \bE |\nabla f|\,.
\end{equation}

So we have for function as above (that is having only two values)
$$
 \frac{I(p)}{2p(1-p)} \bE|f-\bE f| = 2p(1-p) \frac{I(p)}{2p(1-p)} =I(p) \le  \bE |\nabla f|\,,
$$
or
\begin{equation}
\label{2values}
\bE|f-\bE f| \le \max_{p\in [0,1]} \frac{2p(1-p)}{I(p)} \,\,\bE|\nabla f| \le  \sqrt{\frac{\pi}{2}} \,\,\bE|\nabla f|\,.
\end{equation}

 So if the $L^1-$Poincare 
inequality were to not hold for general functions with the optimal 
constant, that begs the question what extremizers could possibly look 
like: then they cannot look like indicators, as they do in the Gaussian case.

Of course, in the continuous case one can re-derive the $L^1$-Poincar\'e 
inequality from its set version, but this does not work in the discrete 
case as it requires the co-area formula.

\bigskip

\subsection{Lipschitz properties in Gaussian setting}
\label{lip}

In this subsection let $0\le f\le 1$.
In Section \ref{BellmanG} we have seen that the following holds in Gaussian setting for the Ornstein--Uhlenbeck semi-group $P_t$ and Bobkov's function $I= \Phi'\circ\Phi^{-1}$:
\begin{equation}
\label{OUPtf}
  I(P_t f)^2 - (P_tI(f))^2 \ge (e^{2t}-1) |\nabla P_t f|^2. \quad \nabla P_sf = e^{-s} P_s\nabla f\,.
\end{equation}
For our purposes, we ignore the second term on the left. That is, we are 
interested in the following slightly weaker inequality:
\begin{equation}
\label{OU1}
  I(P_t  f)^2\ge (e^{2t}-1) |\nabla P_t f|^2.     
\end{equation}
As is already remarked by Bakry-Ledoux, this inequality has the following 
equivalent formulation:
\begin{equation}
\label{OU2}
  (e^{2t}-1) |\nabla \Phi^{-1}(P_tf)|^2 \le 1.   
\end{equation}

Indeed, this follows immediately from the chain rule. (Note that the 
equivalence between \eqref{OU1} and \eqref{OU2} does not hold on the hypercube where the 
chain rule does not hold; so not clear which is more natural.)

In other words, in the Gaussian case, the estimate we seek has a very 
clean reformulation: the quantity $\Phi^{-1}(P_tf)$ (which is precisely what 
appears, say, in Ehrhard inequality) is Lipschitz with universal constant 
depending only on $t$. This estimate is very useful, e.g. it was used it in 
the characterization of equality cases of Ehrhard inequality \cite{RvHE}.

Now we want to give another formulation of \eqref{OU2} that is even more basic. We
claim that \eqref{OU2} should be viewed as a sort of dual isoperimetric inequality 
for Gaussian measure. 

We have not seen discussion of such inequalities in 
the literature but it surely seems natural. To be precise, we claim \eqref{OU2} is 
equivalent to the following extremal statement:
among all functions $0\le f\le1$, the quantity $|\nabla \Phi^{-1}(P_tf)|$
  is maximized pointwise when $f=1_H$, where $H$ is a half-space.   

Indeed it suffices to note that equality in \eqref{OU2} holds pointwise whenever 
$f=1_H$ is any half-space. This shows both that the inequality is sharp and 
that it has a sort of isoperimetric interpretation.

It is not at all clear what the analogous considerations might be on the 
hypercube.

\section{Paradoxical experiments}
\label{para}

 Let us consider again the $L^1$-estimate on the cube: $ \bE |f-\bE f | \le C \| \nabla f \|_1$.
 Define $g:\, \mathbb R \to \mathbb R$ such that  $g(z)=1$ for $z> 0$, $g(z)=0$ for $z=0$  and
 $g(z)=-1$ for $z<0$. Let $f_n(x_1,\cdots, x_n)= g\Big( \frac { \sum_{j=1}^n x_j} {\sqrt n}\Big)$.  
 
 \medskip
 
 Consider first
 $n\gg 1$ with $n$ being odd.  
 
 Obviously $\bE f_n=0$ and $\bE |f_n|=1$. On the other hand,
 for each $i=1,\cdots, n$, easy to check that $|\partial_i f_n|$ takes  value only $0$ or $1$.
 
 \medskip

 Furthermore, denoting $Z_n= \sum_{j=1}^n x_j$, it is not difficult to check that
 $|\partial_i f_n|=1$ if and only if  either  $Z_n=1$, $x_i=1$, or $Z_n=-1$, $x_i=-1$.
 
 From this we get $|\nabla f_n|=\sqrt
{\frac {n+1}2}$ or $0$, and the number of vertices where $|\nabla f_n|\neq 0$ is precisely $2 \binom{n} {\frac {n+1}2}$. Therefore,
 \begin{align}
 \label{2sqrtpi}
 \| \nabla f_n \|_1 =  \frac 1 {2^n} \cdot \binom{n} {\frac {n+1}2} \cdot
 \sqrt{\frac {n+1}2} \cdot 2  ={\frac 2 {\sqrt{\pi}}} \cdot (1+o_n(1)),
 \end{align}
 as $n$ tend to infinity. 
 
 \medskip
 
 Now consider $n\gg 1$ with $n$ being even.  Then $\bE f_n=0$
 and $\bE |f_n|=1-o_n(1)$.  On the other hand, now we will be jumping from $\pm 1$ values to $0$ values while calculating $\pd_i g$, hence:
 \begin{align*}
 \|\nabla f_n\|_1 = 2^{-n} 
 \cdot 
 \binom{n}{\frac n2} \cdot \sqrt n \cdot \frac 12
 + 2^{-n}
 \cdot \binom{n}{\frac n2+1} 
 \cdot \sqrt{\frac n2 +1} \cdot \frac 12 \cdot 2 = 
 \frac {1+\sqrt 2}{\sqrt {2\pi}} \cdot (1+o_n(1)).
 \end{align*}
 
 \medskip
 
 The above two cases show that for $2$-valued $g$ one cannot saturate the optimal constant for
 the discrete Hamming cube case.  For the odd $n$ case we get the constant
 $$
 C_{odd, charact. function} =\frac{\sqrt{\pi}}{2} <\sqrt{\frac{\pi}{2}},
 $$
 and for the even $n$ case we get
 $$
 C_{even, charact. function} =\sqrt{\pi}\frac{\sqrt{2}}{\sqrt{2}+1}<\sqrt{\frac{\pi}{2}}.
 $$

 Amusingly, if we take $g$ to be a smooth function, then it
 is not difficult to check that for $f_n= g( (\sum_{j=1}^n x_j)/\sqrt n)$, one has
 \begin{align*}
 |\partial_j f_n| = \frac 1 {\sqrt n} \cdot \Big( \big|g^{\prime}\Big( \frac {\sum_{j=1}^n x_j} {\sqrt n} \Big)\big| 
 +O(n^{-\frac 12})\Big).
 \end{align*}
 From this one gets
 \begin{align}
 \label{smooth}
 \frac {\bE |f_n- \bE f_n |} {\bE |\nabla f_n|} \to \frac{ \bE_{\gamma} |g (z)-\bE_\gamma g(z) |} {\bE_{\gamma} |g^\prime|},\quad \text{$n \to \infty$},
 \end{align}
 where $E_{\gamma}$ denotes expectation with respect to standard Gaussian measure on $\mathbb R$.

 \medskip
 
 In particular, choosing function $g$ to be a smooth approximation to ${\bf 1}_{\bR_+}$ and then choosing $f_n$, $f_n(x_1,\cdots, x_n)= g\Big( \frac { \sum_{j=1}^n x_j} {\sqrt n}\Big)$, we conclude that the right hand side of \eqref{smooth} is as close to $\sqrt{\frac{\pi}{2}}$ as we wish. Then making $n\to \infty$ we achieve that the left hand side is also as close to $\sqrt{\frac{\pi}{2}}$ as we wish. 
 
 But these $f_n$ will have values $-1, 1$ and many values in between. it is possible to prove that we cannot achieve 
 the constant $\sqrt{\frac{\pi}{2}}$ by testing symmetric functions having values $-1,1$ or $-1, 0, 1$ alone. 
 The constants in $L^1$-Poincar\'e inequality for such functions are uniformly in $n$ strictly smaller than $\sqrt{\frac{\pi}{2}}$.
 
 \bigskip
 
 \section{Symmetric functions}
 \label{symm}
 
 Let us consider functions having only two values, but symmetric. It is convenient to think now that functions have only values $0,1$, and let us 
 consider balanced functions:
 $$
 \bE f=\frac12\,.
 $$
 Let us now think that $x_i$ are independent standard $0,1$ Bernoulli  random variables.
 Function $f$ has the same value on $R_k:=\{x: x_1+\dots x_n =k\}$. In the previous section we considered the case, when $f$ had one value on 
 all $R_k$ with $k<\frac{n}2$ and another value on all $R_k$, $k> \frac{n}2$ (for $n$ odd, say).
 
 Now let us consider more general symmetric function. As always, being balanced, it will have $\bE |f-\bE f| =\frac12$, so we need to minimize $\bE|\nabla f|$, or, to minimize
 $\bE |\nabla {\bf 1}_A|$, with $|A|=\frac12$. Clearly, it is better for us not to allow ${\bf 1}_A$ to oscillate in too many places. One place of oscillation was considered in the previous section. 
 
 Let us show now that by choosing two places of oscillation we can only make $|A|/\bE |\nabla {\bf 1}_A|$ smaller by making $\bE |\nabla {\bf 1}_A|$ bigger.
 
 So choose $a=\Phi^{-1}(\frac34)$ and put $k =\big[\frac{n}2 + a\frac12 \sqrt{n}\big]$. Let $A$ be the set where
 $x_1+\dots x_n \in [n-k, k]$.
 
 Then 
 $$
 \bE |\nabla {\bf 1}_A| \approx 2 \Big[ \frac12 \sqrt{n-k}\binom{n}{k} \frac1{2^n} +  \frac12 \sqrt{k+1}\binom{n}{k+1} \frac1{2^n} \Big]\,.
 $$
 
 By de Moivre--Laplace formula 
 $$
 \frac1{2^n}\binom{n}{k} \approx \frac{2}{\sqrt{n}} \phi (a) = \frac{2}{\sqrt{n}}  I(\frac34)\,.
 $$
 
  Hence, since $\sqrt{k}\approx \frac{\sqrt{n}}{\sqrt{2}}$ and $\sqrt{n-k}\approx \frac{\sqrt{n}}{\sqrt{2}}$, we get the following:
  $$
 \bE |\nabla {\bf 1}_A| \approx 2\sqrt{2}\, I\big(\frac34\big)\,.
 $$
 Hence, for $f={\bf 1}_A$, with $A$ described above, we have
 $$
 \bE|f-\bE f| \le \frac{1}{4\sqrt{2}\, I\big(\frac34\big)} \bE|\nabla f|\,.
 $$
Constant $C_{odd, charact. function}=\frac{\sqrt{\pi}}{2}$ from the previous section is nothing else than
$\frac{1}{2\sqrt{2}\, I(\frac12)}$. Clearly
$$
\frac{1}{4\sqrt{2}\, I\big(\frac34\big)} = \frac{1}{4\sqrt{2}\, I\big(\frac14\big)}< \frac{1}{2\sqrt{2}\, I(\frac12)},
$$
because $2I\big(\frac14\big)> 2 I(\frac12)$ by concavity of function $I$. 

The conclusion: the symmetric function in this section gives a smaller constant in $L^1$-Poincar\'e inequality than a simpler characteristic function in the 
previous section. It is very believable that among balanced symmetric  functions with only two values  it is the optimal one, thus, 
$$
\max\frac{\bE|f-\bE f|}{\bE|\nabla f|} = \frac{\sqrt{\pi}}{2}\,.
$$

\begin{remark}
Consider this maximum over all functions having two values \textup(without the loss of generality, 
just values $0,1$\textup). We saw that it is  at most $\sqrt{\frac{\pi}{2}}$ in this general setting. But is this constant attained?
For balanced symmetric functions, it is now very believable \textup(by the discussion 
in the present section\textup) that maximum above is much smaller, 
namely $\frac{\sqrt{\pi}}{2}$. But even for all functions with two values \textup(with no symmetries whatsoever\textup) that maximum can be smaller than $\sqrt{\frac{\pi}{2}}$. 
In Proposition 3.1 on page 259 of \cite{BobG} functions $f={\bf 1}_{A_n(a)}$ are considered. Here $a\in \bR$, and 
$$
A_n(a)=\Big\{x: \frac{x_1+\dots +x_n -\frac{n}2}{\frac12\sqrt{n}} \le a\Big\},
$$
where $x_i$ are standard independent Bernoulli variables with values $0,1$. If $a=\Phi^{-1}(\alpha), \alpha\in [0,1]$, then
$$
\bE |{\bf 1}_{A_n(a)} -\bE{\bf 1}_{A_n(a)}| =2 \alpha(1-\alpha)\,.
$$
On the other hand, $\bE|\nabla {\bf 1}_{A_n(a)}|$ is calculated on page 259 of \cite{BobG}:
$$
\bE|\nabla {\bf 1}_{A_n(a)}|\approx \sqrt{2}\,\phi(a) =\sqrt{2}\, I(\alpha)\,.
$$
Whence,
\begin{equation}
\label{Ana}
\lim_{n\to\infty}\max_{\alpha\in [0,1]}\frac{\bE |{\bf 1}_{A_n(a)} -\bE{\bf 1}_{A_n(a)}|}{\bE|\nabla {\bf 1}_{A_n(a)}|} =\frac{1}{2\sqrt{2}\, I(\frac12)} =\frac{\sqrt{\pi}}{2}, \quad a=\Phi^{-1}(\alpha), \alpha\in [0,1]\,.
\end{equation}
\end{remark}

\begin{remark}
We saw above function a certain $f$ having three values $1, 0, -1$, for which the ratio $\frac{\bE|f-\bE f|}{\bE|\nabla f|}$ is bigger than $\frac{\sqrt{\pi}}{2}$, namely it is asymptotically $\frac{\sqrt 2}{\sqrt 2+ 1} \sqrt{\pi}$. We also saw that allowing more values we can saturate the constant $\sqrt{\frac{\pi}{2}}$. It is not clear whether we can surpass this constant.
\end{remark}

\section{Kernel representation of operator  $T$}
\label{T}
Operator $f\to \int_0^\infty \nabla e^{t\Delta} \, f \,dt$ can be written also as follows:
$$
Tf =  \int_0^\infty e^{-t} e^{t\Delta} \, \nabla f \,dt=\int_0^\infty \nabla e^{t\Delta} \, f \,dt\,.
$$
Notice that we cannot loose $e^{-t}$ here, if we drop $e^{-t}$, the expression will become undefined for, say, $f=x_1$. We cannot either consider anything like
$$
f\to \nabla\int_0^\infty e^{t\Delta} f\, dt =\nabla\Delta^{-1}\, f,
$$
because this latter expression is undefined on $f={\bf 1}$.

However, we can remedy this drawback just by introducing the orthogonal projection $P_0$ onto functions on Hamming cube that have average zero: $P_0 f:=f-\bE f$. Then  we can write down $T$ in the following form
\begin{equation}
\label{TP0}
T f = \nabla \Delta^{-1} P_0 f\,.
\end{equation}

\bigskip

But for operator $T$ defined above, it is easy to give its matrix (kernel) representation.
For that let us double the cube: $C^{2n}=\{(x', x)\in C^n \times C^n\}$ and consider
$$
\Pi (x', x) := \Pi_t(x', x):=\Pi_{k=1}^n (1+e^{-t} x_k' x_k)\,.
$$

If $d\mu(x')$ is the uniform measure on the first $C^n$ and if $x$ in the second $C^n$ get fixed, we get a new probability measure 
$$
d\bP (x'):= d\bP_x(x'):= \Pi(x', x)\, d\mu(x')\,.
$$
It is very easy to see that it is indeed a probability measure for any $x$.

In these terms it is easy to compute the matrix of operator $T$.
In fact, we have already done this above: let $K(x', x)$ be the kernel representing $T$ in the sense that
$$
Tf(x) = \int_{C^n} K(x', x) f(x') \, d\mu(x')\,.
$$
It is a vector kernel, and let $T_i$ corresponds to $\int_0^\infty \pd_i e^{t\Delta} \, f \,dt$, where $\pd_i$ is the elimination operator for $x_i$. Let $K_i$ be the kernel of $T_i$ in the just mentioned sense.
Then
\begin{equation}
\label{KT}
K_i(x', x) = \int_0^\infty\frac{e^{-t}}{\sqrt{1-e^{-2t}}}\frac{x_i'-e^{-t} x_i}{\sqrt{1-e^{-2t}}} \Pi(x',x) \, dt\,.
\end{equation}

This we already had in the third Section essentially. Now let us rewrite it conveniently.
\begin{equation}
\label{KT1}
K_1(x', x) = \int_0^\infty\frac{e^{-t}}{1-e^{-2t}}(x_1'-e^{-t} x_1)\,\Pi(x', x)\, dt\,.
\end{equation}

This does not look very nice because of $t\approx 0$ seems like creating a problem. But it does not, because this expression can be rewritten as follows:
\begin{equation}
\label{KT1a}
K_1(x', x)  = \int_0^\infty\frac{e^{-t}}{1-e^{-2t}} x_1'\,(1-e^{-t} x_1'x_1)(1+e^{-t} x_1'x_1) \Pi_{k=2}^n (1+e^{-t} x_k' x_k)\, dt\,,
\end{equation}
which is 
\begin{equation}
\label{KT1b}
K_1(x', x)  = \int_0^\infty\frac{e^{-t}}{1-e^{-2t}} x_1'\,(1-e^{-2t})\Pi_{k=2}^n (1+e^{-t} x_k' x_k)\, dt\,,
\end{equation}
which is
\begin{equation}
\label{KT1c}
K_1(x', x)  = \int_0^\infty e^{-t} x_1'\,\Pi_{k=2}^n (1+e^{-t} x_k' x_k)\, dt\,,
\end{equation}

Now the kernel of $T$ is $K=(K_1, K_2, \dots, K_n)$, where $K_i$ is
\begin{equation}
\label{KTi}
K_i(x', x)  = \int_0^\infty e^{-t} x_i'\,\Pi_{k\neq i}^n (1+e^{-t} x_k' x_k)\, dt\,,\quad i=1,\dots, n\,.
\end{equation}

\bigskip

Recall the notation $C^n =\{-1,1\}^n$.
We are interested in the norm of the integral operator with kernel $K=(K_1,\dots, K_n)$ as the operator from $L^\infty (C^n)\to L^\infty(C^n; \ell^2_n)$. 

Given $x, y\in C^n$ we write $z=y\cdot x=(y_1x_1,\dots, y_nx_n)\in C^n$.

\bigskip

Kernel $K_i(x', x)$ can be written down as $K_i(x', x)=x_i \tilde K_i(x', x)$, where
$$
 \tilde K_i(x', x):= \int_0^\infty e^{-t} x_i'x_i\,\Pi_{k\neq i}^n (1+e^{-t} x_k' x_k)\, dt,
 $$
and we see that this is a kernel of the form $k_i(x'\cdot x_i)$, that is, it is a convolution kernel.

We use it with measure $d\mu(x')$ that is invariant, meaning that $d\mu(x'\cdot x)$ is the same measure.
Notice also that  the facts that $K_i(x', x)=x_i \tilde K_i(x', x), x_i=\pm 1,$  imply that the norm of operator with kernel $K$ is the same as the norm of operator with kernel $\tilde K$ -- we mean here the action from $L^\infty (C^n)\to L^\infty(C^n; \ell^2_n)$.

Let us write
\begin{equation}
\label{miz}
m_{i, z}:=2^{-n}K_i(z, {\bf 1})  = 2^{-n}\int_0^\infty e^{-t} z_i\,\Pi_{k\neq i}^n (1+e^{-t} z_k)\, dt=2^{-n}\int_0^1z_i\,\Pi_{k\neq i}^n (1+\rho z_k)\, d\rho\,,\quad i=1,\dots, n\,,
\end{equation}
The reasoning that we have just made implies that the norm of the integral operator $K$ is the same as the norm of 
the matrix $M:= (m_{i,z})_{i=1,\dots, n;\,z\in \{-1,1\}^n}$ as the matrix acting from $\ell^{\infty}_{2^n}$ to $\ell^2_n$. In fact, this is just invariance:
$$
\int_{C^n} \tilde K(x'\cdot x) f(x') d\mu(x')= \int_{C^n} \tilde K(y) f_x(y) d\mu(y), \quad f_x(y):= f(y\cdot x)\,.
$$

Since $f\to f_x$ is an isometry in  $L^\infty(C^n)$, we see that the norm of our operators from 
$L^\infty (C^n)$ to $ L^\infty(C^n; \ell^2_n)$ is the same as the norm of matrix $M$ from $\ell^\infty_{2^n}$ to $\ell^2_n$.
\bigskip

Rewriting again, we get:
$$
K_i(z, {\bf 1})  =\int_0^1\frac{z_i(1-\rho z_i)}{1-\rho^2}\,\Pi_{k=1}^n (1+\rho z_k)\, d\rho= \int_0^1\frac{z_i-\rho }{1-\rho^2}\,\Pi_{k=1}^n (1+\rho z_k)\, d\rho\,,\quad i=1,\dots, n\,,
$$

Let us denote by $d(z)=\dist(z, {\bf 1})$ in Hamming metric. Then for $z\in \{-1, 1\}^n$ we have:
$$
\Pi_{k=1}^n (1+\rho z_k) = (1+\rho)^{n-d(z)} (1-\rho)^{d(z)}\,.
$$

Thus,
\begin{equation}
\label{KTiz}
K_i(z, {\bf 1})  = \int_0^1\frac{z_i-\rho }{1-\rho^2}(1+\rho)^{n-d(z)} (1-\rho)^{d(z)}\, d\rho\,,\quad i=1,\dots, n\,,
\end{equation}

Let $M_\rho:= (m_{i,z}(\rho))_{i=1,\dots, n;\,z\in \{-1,1\}^n}$ be the matrix acting from $\ell^{\infty}_{2^n}$ to $\ell^2_n$, whose $n\times 2^n$ matrix elements are given by the following formula:
\begin{equation}
\label{me}
m_{i,z}(\rho) := 2^{-n}\frac{z_i-\rho }{1-\rho^2}(1+\rho)^{n-d(z)} (1-\rho)^{d(z)}\,, \quad i=1, \dots, n, z\in\{-1,1\}^n\,, \rho\in (0,1)\,.
\end{equation}

\bigskip

Of course, matrix elements $m_{i,z}:=2^{-n}K_i(z,{\bf 1})$ are just $\int_0^1 m_{i, z}(\rho)\, d\rho$, and one may wonder what to do with integration of $1/(1-\rho^2)$? But this is easy: if $d(z)>0$ then we cancel this singularity by $(1-\rho)^{d(z)}$ factor, and if $d(z)=0$, then of course $z={\bf 1}$ and instead of factor $1/(1-\rho^2)$ we have factors $(z_i-\rho)/(1-\rho^2)= (1-\rho)/(1-\rho^2)=1/(1+\rho)$ for all $i$.

To compute the norm of the matrix $M=(m_{i,z})=\int_0^1 M_\rho\, d\rho$ as the matrix acting from $\ell^{\infty}_{2^n}$ to $\ell^2_n$ one can try to do two different things.

\bigskip

\noindent{\bf The first attempt.} Calculate
\begin{equation}
\label{FA}
\max_{\la,\, \|\la\|_{\ell^2_n}\le 1} \bE_z\Big| \sum_{i=1}^n \la_i \int_0^1\frac{z_i-\rho }{1-\rho^2}(1+\rho)^{n-d(z)} (1-\rho)^{d(z)}\, d\rho\Big|
\end{equation}

Or one can try to use a rougher estimate as follows.

\bigskip

\noindent{\bf The second attempt.} Calculate
\begin{equation}
\label{SA}
\int_0^1\max_{\la,\, \|\la\|_{\ell^2_n}\le 1} \bE_z\Big| \sum_{i=1}^n \la_i \frac{z_i-\rho }{1-\rho^2}\Big|\,(1+\rho)^{n-d(z)} (1-\rho)^{d(z)}\,  d\rho
\end{equation}

The second attempt is precisely what we have done in Sections \ref{mala} and \ref{close}, especially see \eqref{3norms}. We were not very careful in estimating the quantity in \eqref{SA}, we just proved that it is strictly smaller than $\pi/2$.

\bigskip

However, the integrals in \eqref{FA}, namely, 
$$
m_{i, z}=\int_0^1\frac{z_i-\rho }{1-\rho^2}(1+\rho)^{n-d(z)} (1-\rho)^{d(z)}\, d\rho
$$
seems to be treatable, they can be written down as certain combinatorial sums.

\bigskip

For example,
$$
m_{i, {\bf 1}} = \int_0^1 (1+\rho)^{n-1}\, d \rho, \quad i=1, \dots, n\,.
$$
For $z=(1,\dots, 1, -1)$ we have
$$
m_{i, z} = \int_0^1 (1+\rho)^{n-2}(1-\rho)\, d \rho, \quad i=1, \dots, n-1\,,
$$
$$
m_{n, z} = - \int_0^1 (1+\rho)^{n-1}\, d \rho\,.
$$

\begin{figure}[!ht]
\hskip-240pt\vbox{\includegraphics[scale=0.85]{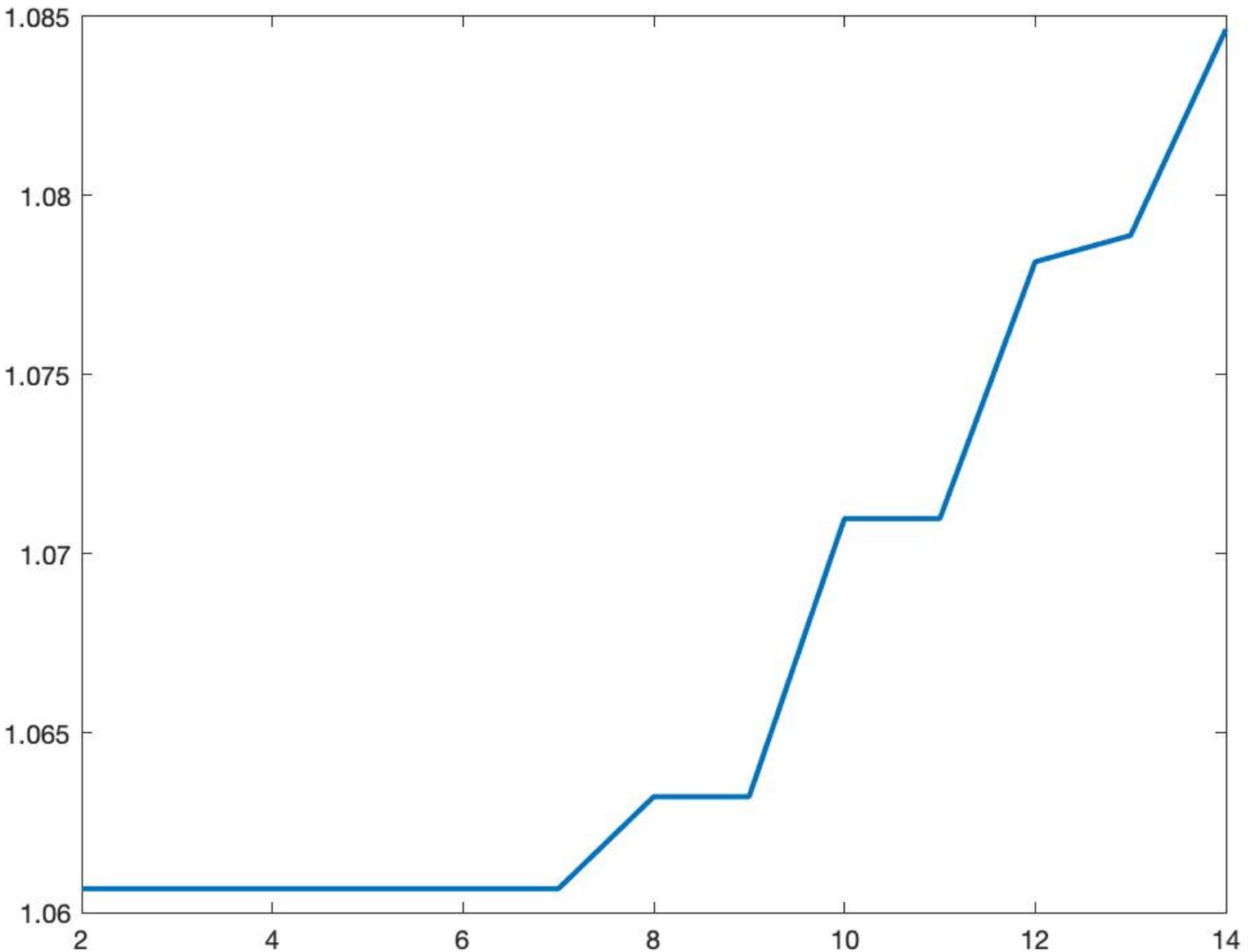}}
\caption{Growth of dual constant with $n$.}
\label{Dual constant}
\end{figure}

\section{Computer experiment for finding $C_{dual}$}
\label{coex}

Let us consider the quantity
 \begin{equation}
\label{Cdual0}
 C_{dual, n}:= \sup_{\|g\|_{L^\infty(\{-1,1\}^n) }\le 1}\|\int_0^\infty  \nabla P_t  g\, dt \|_{\infty}\,,
 \end{equation}
 and let us try to find its numerical values for small dimensions $n$.
 
 \medskip
 
For $n=2$  the optimizer is
$$
g(x_1, x_2) = \min(x_1,x_2) = (x_1 x_2 + x_1 + x_2 - 1)/2.
$$
In dimensions $n=3,\dots,7$, the optimal value is the same so one can just 
take two-dimensional function like $g(x_1,\dots,x_n) = \min(x_1,x_2$). 
Of course, there may be many other optimizers.

The first nontrivial dimension where the $2D$ case is not optimal is $n=8$.

We would like to call the attention of the reader that the graph we get is extremely curious. It 
seems up to dimension $7$, two-dimensional functions are optimal. Then 
suddenly in dimension $8$ there is enough structure to do better with a 
truly eight-dimensional function. 

The optimal value increases only very 
slowly. If we assume it grows sort of linearly (no reason it should, just 
to assume it to make some guesses), then one would extrapolate from the plot 
that $C_{dual,n}$ would reach $\sqrt{\pi/2}$ only around dimension $n=60$. This 
means that there is probably little insight to be gained from the specific 
structure of low-dimensional optimizers. Also, it is very different than the usual experience, which is that universality phenomena (like CLT) often 
kick in at surprisingly low dimension. For example, for Bernoulli $\eps_k$ one has
$$
\bE|\eps_1+...+\eps_9|/\sqrt{9}| \approx 0.82,\,\,
\bE|\eps_1+...+\eps_{13}|/\sqrt{13}| \approx0.81,
$$
which is quite close to the Gaussian limit $\sqrt{2/\pi} \approx 0.80$. Here, we are 
very far from the Gaussian case and seem to approach it only very slowly.

 \medskip

See Fig.~\ref{Dual constant} for the growth of the dual  constant with $n$.
One can see that it grows very slowly (and does not grow at all for $n=2, \dots, 7$). Then it slowly starts to pick up.
Since we work with matrices of size $n\times 2^n$, their size grows too fast. We have already noted that the growth on this figure suggests that we come close to $\sqrt{\pi/2}$ only for $n=60$ or $70$. This experiment is beyond the computer reach.

\section{$Curl$ space and $C_{dual}$}
\label{curl}

Let us remind the reader that
\begin{align*}
&(f - \bE f, g) =   - ( \nabla f ,  \int_0^\infty  \nabla P_t  g\, dt) \,.
\end{align*}

Therefore,
\begin{align*}
&\bE |f-\bE f| \le \|\nabla f\|_1 \cdot \sup_{\|g\|_\infty \le 1} \Big\| \int_0^\infty  \nabla P_t  g\, dt\Big\|_{L^\infty/Curl} \le
\end{align*}

So we need to estimate
\begin{equation}
\label{C1}
C_1:=  \sup_{\|g\|_\infty \le 1}\inf_{h\in Curl}\|h+ \int_0^\infty  \nabla P_t  g\, dt \|_{\infty}\,,
 \end{equation}
 where the ``$Curl$" space is the following:
 $$
 Curl:=\{ h=(h_1,\dots, h_n): \bE( h\cdot \nabla \varphi)= 0\,, \forall \varphi\}\,.
 $$
 
 Since $\nabla=(\pd_1, \dots, \pd_n)$, and $\pd_k$ is the elimination of $x_k$ operator:
 $$
 \pd_k \varphi = \frac12(\varphi^{x_k\to 1} - \varphi^{x_k\to -1})\,.
 $$
 
 \medskip
 
 Space $Curl$ consists of vector functions $h=(h_1,\dots, h_n)$, such that
 \begin{equation}
 \label{stars}
 \pd^*_1 h_1 +\dots +\pd^*_n h_n=0\,.
 \end{equation}
 
 Here
 $$
 \pd^*_k \varphi = x_k \frac12(\varphi^{x_k\to 1} + \varphi^{x_k\to -1})\,.
 $$
 In other words, it is a creation operator, that is
 $$
 \pd^*_k x^S =\begin{cases}  x_k x^S, \,\, \text{if}\,\, k\notin S, 
 \\
 0,\,\,\,\,\, \,\,\,\,k\in S\,.
 \end{cases}
 $$
 
  Space $Curl$ is very large, as the vector functions $h$ such that $\pd^*_k h_k=0, k=1, \dots, n,$ are in this space. And for that to hold, it is enough for each  $h_k$ to be of the following form:
  \begin{equation}
  \label{hk}
  h_k = x_k\pd_k H_k
  \end{equation}
  with arbitrary $H_k$, $k=1, \dots, n$. In fact, space $Curl$  is much larger than that.
  
  \medskip
  
  To get the value of $C_1$ is the same as to calculate the quantity in \eqref{C1}.
  In the previous sections we gave some estimates on a potentially bigger quantity $C_{dual}$, which is given by the following formula:
  \begin{equation}
\label{Cdual}
 C_{dual}:= \sup_{\|g\|_\infty \le 1}\|\int_0^\infty  \nabla P_t  g\, dt \|_{\infty}\,,
 \end{equation}
 
 \bigskip
 
 \subsection{$Curl$ space in gaussian setting does not matter}
 \label{curlgauss}
 
Consider functions in $L^1(\bR^1, \gamma_1)$ orthogonal to all  $\cD:=\{ f': f\in C_0^\infty(\bR^1)\}$.
Then
$$
\int h f' d\gamma_1=0,\quad\forall f' \in \cD\,.
$$
Therefore,
$$
-\int h' f d\gamma + \int xh f d\gamma_1=0,\quad \forall f\in C_0^\infty\,.
$$
Hence
$$
h'=xh,\quad h=C\cdot e^{\frac{x^2}2}\,.
$$
This does not belong to $L^1(\gamma)$ unless $C=0$. So only zero function is in $Curl$. So in gaussian setting the $Curl$ space is zero.

But in dimension $2$ and higher curl space unfortunately exists in gaussian setting. In fact, in $2D$ this space consists of vector functions $h=(h_1, h_2 )\in  L^1(\bR^2, \gamma_2)$ such that
\begin{equation}
\label{cg2D}
(h_1)_{x_1} + (h_2)_{x_2} = x_1 h_1 + x_2 h_2\,,
\end{equation}
which has a solution $h_1=-x_2, h_2=x_1$.

\medskip

In gaussian space, however, we know that 
\begin{equation}
\label{C1eqCdual}C_1= C_{dual}\,,
\end{equation}
 where these constants can be seen in \eqref{C1} and \eqref{Cdual} correspondingly -- where $P_t$ should be understood as Ornstein--Uhlenbeck semi-group. 

For $n=1$ this follows from the above mentioned fact that $Curl=0$ for $1D$ gaussian case.  
Now let $n>1$. Let $G$ be the function of one variable that almost give the supremum in

$$
\sup_{\|g\|_\infty \le 1}\|\int_0^\infty  \nabla P_t  g\, dt \|_{\infty}
 $$
 for $n=1$.  Whence,
 $$
   \|\int_0^\infty  \nabla P_t^{(1)}  G\, dt \|_{\infty}= \|\int_0^\infty  e^{-t}P_t^{(1)}   \nabla G\, dt \|_{\infty}\approx \sqrt{\frac{\pi}{2}}\,.
  $$
 Here $P_t^{(1)}$ is  Ornstein--Uhlenbeck semi-group in $L^1(\bR^1, \gamma_1)$. Since function $G=G(x_1)$, we can understand the last inequality also with $P_t^{(n)}$ is  Ornstein--Uhlenbeck semi-group in $L^1(\bR^n, \gamma_n)$:
$$
   \|\int_0^\infty  \nabla P_t^{(n)}  G\, dt \|_{\infty}= \|\int_0^\infty  e^{-t}P_t^{(n)}   \nabla G\, dt \|_{\infty}\approx \sqrt{\frac{\pi}{2}}\,.
  $$
  This is because $(P_t^{(n)} g)(x) = (P_t^{(1)} g)(x_1)$ for $g$ depending only on $x_1$.

  \section{Combinatorial formulation of $C_{dual}$}
\label{combi}

Constant $C_{dual}$, given by $ C_{dual}:= \sup_{\|g\|_\infty \le 1}\|\int_0^\infty  \nabla P_t  g\, dt \|_{\infty}$ is just 
$$
C_{dual} = \|T\|_{L^\infty\to L^{\infty}}\,,
$$
where $T$ is, e.g. from \eqref{TP0}, that is, $T=\nabla \Delta^{-1}P_0$.

The norm $ \|T\|_{L^\infty\to L^{\infty}}$ is the smallest constant $C$ in the following inequality:
$$
\|\nabla \Delta^{-1} P_0 f\|_\infty  \le C\|f\|_\infty,
$$
which can be written down as follows
$$
\|\nabla \Delta^{-1}  f_0\|_\infty  \le C\inf_a\|f_0 + a{\bf 1}\|_\infty\,,
$$
where $f_0$ runs over all functions on Hamming cube that have zero average. Such functions can be written down as $f_0 = \Delta F$, so we plug this representation into the above formula. Notice that $\Delta^{-1}\Delta F = F- \bE F$. So we are looking at the best constant  $C$ in

$$
\|\nabla F\|_\infty =\|\nabla (F-\bE F)\|_\infty  \le C\inf_a\|\Delta F+ a{\bf 1}\|_\infty\,.
$$

Denote Hamming graph as $(V, E)$, where vertices are denoted by $i=1, \dots, 2^n$.

Then the previous best constant squared, namely, $C^2$ (that, is $C_{dual}^2$) is the best constant in the following inequality with arbitrary real numbers $\{a_i\}_{i=1}^{2^n}$:
\begin{equation}
\label{combi-eq}
\sup_{i\in V} \sum_{j: (i, j) \in E} (a_i-a_j)^2 \le C^2 \inf_{a\in \bR} \sup_{i\in V}\Big( a+ \sum_{j: (i, j)\in  E} (a_i- a_j)\Big)^2\,.
\end{equation}

In particular, we have proved that with $\eps>0$ and independent of $n$
\begin{equation}
\label{combi-eq1}
\sup_{i\in V} \sum_{j: (i, j) \in E} (a_i-a_j)^2 \le \Big(\frac{\pi}{2} -\eps\Big)^2 \sup_{i\in V}\Big( \sum_{j: (i, j)\in  E} (a_i- a_j)\Big)^2\,.
\end{equation}

\section{Calculations with matrix $M$. An example when $Curl$ space is essential}
\label{M}

Recall that we introduced in  \eqref{me} the following  $n\times 2^n$ matrix $M_\rho:= (m_{i,z}(\rho))_{i=1,\dots, n;\,z\in \{-1,1\}^n}$:
\begin{equation}
\label{me1}
m_{i,z}(\rho) := 2^{-n}\frac{z_i-\rho }{1-\rho^2}(1+\rho)^{n-d(z)} (1-\rho)^{d(z)}\,, \quad i=1, \dots, n, z\in\{-1,1\}^n\,, \rho\in (0,1)\,.
\end{equation}

We considered it as acting  from  $\ell^{\infty}_{2^n}$ to $\ell^2_n$.

\bigskip

Of course, matrix elements $m_{i,z}:=2^{-n}K_i(z,{\bf 1})$ are just $\int_0^1 m_{i, z}(\rho)\, d\rho$.  Again we considered it as acting  from  $\ell^{\infty}_{2^n}$ to $\ell^2_n$, and we know that the sharp dual constant in $L^1$-Poincar\'e inequality is the norm of this matrix as acting from  $\ell^{\infty}_{2^n}$ to $\ell^2_n$.

The norm of the matrix $M=(m_{i,z})=\int_0^1 M_\rho\, d\rho$ as the matrix acting from $\ell^{\infty}_{2^n}$ to $\ell^2_n$ is
\begin{equation}
\label{FA1}
\max_{\la,\, \|\la\|_{\ell^2_n}\le 1} \bE_z\Big| \sum_{i=1}^n \la_i \int_0^1\frac{z_i-\rho }{1-\rho^2}(1+\rho)^{n-d(z)} (1-\rho)^{d(z)}\, d\rho\Big|
\end{equation}

We know that it is less than $\frac{\pi}2 -\eps$, where $\eps>0$ does not depend on $n$.

\bigskip

There is a ``sister" problem, where $L^1$-Poincar\'e inequality is replaced by $L^\infty$ one. It cannot have the form
$\|f-\bE f\|_\infty \le C \|\nabla f\|_\infty$ with constant independent of $n$. This is impossible, e.g. because inequality
$a_1+\dots + a_n \le C(a_1^2+\dots +a_n^2)^{1/2}$ is false. However, if one changes the definition of the gradient, then the corresponding inequality 
becomes meaningful and important (see, e.g., \cite{FP} or \cite{Wagner}.)

Denote $|\tilde\nabla f |(x)= \sum_{i=1}^n |\pd_i f(x)|$. Then one can ask, whether the following inequality holds and with what sharp constant independent of $n$:
\begin{equation}
\label{Linfty}
\|f-\bE f\|_{\infty}\le C_\infty\|\tilde \nabla f\|_{\infty}\,?
\end{equation}

This is a combinatorial question about the diameter of Hamming cube with weighted lengths of edges. Namely, this is equivalent to asking (see \cite{FP}) what is the supremum of $L^\infty$ norms of functions $f$  having zero average on cube and such that
\begin{equation}
\label{tina}
|\tilde\nabla f(x)| \le 1 \quad \forall x\in \{-1,1\}^n\,.
\end{equation}

\begin{remark}
Condition \eqref{tina} can be reformulated  in purely combinatorial terms as follows \textup(see \cite{FP}\textup):  every edge of the graph (=Hamming cube) is provided with its variable ``length" $\ell(x, y)$, and
$$
\sum_{y\sim x} \ell(x, y) \le 2\quad \forall x\in \{-1,1\}^n\,.
$$
One wants then to know what is the universal sharp estimate on the diameter of the cube?
\end{remark}
This becomes the question about independent of $n$ sharp constant $C_\infty$ such that \eqref{Linfty} holds.
The reference \cite{FP} has a very nice reasoning of F. Petrov that proves the following:
$$
C_\infty = 2\,.
$$

Notice that we can easily translate the estimate \eqref{Linfty} into a certain fact about our matrix $M=(m_{i,z})=\int_0^1 M_\rho\, d\rho$ defined in \eqref{miz}. In fact, repeating our reasoning in the previous sections,
we can notice that \eqref{Linfty} is equivalent to finding
$$
\sup_{f: \|f||_{L^1(\{-1, 1\}^n} \le 1}\inf_{h\in Curl} \|h+\int_0^\infty \nabla P_t f dt\|_{L^1(\{-1, 1\}^n; \ell^\infty_n)}\,.
$$

If $Curl $ space would not play any role, that would mean that we are interested in 
$$
\sup_{f: \|f||_{L^1(\{-1, 1\}^n} \le 1}\|\int_0^\infty \nabla P_t f dt\|_{L^1(\{-1, 1\}^n; \ell^\infty_n)}\,.
$$

We can calculate this norm now. It is the norm of our familiar operator
$$
f\to \int_0^\infty \nabla P_t f dt
$$
as acting from $L^1(\{-1,1\}^n, d\mu)$ to $L^1(\{-1,1\}^n, d\mu; \ell^\infty_n)$.

\bigskip

This is  the same the norm of operator $K=(K_1, \dots, K_n)$ (introduced in Section \ref{T}) as acting from $L^1(C^n, d\mu)$ to $L^1(C^n, d\mu; \ell^\infty_n)$. 
Kernel $K_i(x', x)$ can be written down as $K_i(x', x)=x_i \tilde K_i(x', x)$, where
$$
 \tilde K_i(x', x):= \int_0^\infty e^{-t} x_i'x_i\,\Pi_{k\neq i}^n (1+e^{-t} x_k' x_k)\, dt,
 $$
and we see that this is a kernel of the form $k_i(x'\cdot x_i)$, that is, it is a convolution kernel.

We use it with measure $d\mu(x')$ that is invariant, meaning that $d\mu(x'\cdot x)$ is the same measure.
Notice also that  the facts that $K_i(x', x)=x_i \tilde K_i(x', x), x_i=\pm 1,$  imply that the norm of operator with kernel $K$ is the same as the norm of operator with kernel $\tilde K$ -- we mean here the action from $L^1 (C^n)\to L^1(C^n; \ell^\infty_n)$.

The norm of  convolution operator from $L^1$ to $L^1$ is just $L^1$ norm of its kernel.

The reasoning that we have just made implies that the norm of the integral operator $K$ from $L^1 (C^n)\to L^1(C^n; \ell^\infty_n)$ is the same as the norm of 
the matrix $M:= (m_{i,z})_{i=1,\dots, n;\,z\in \{-1,1\}^n}$  considered as the vector function in $L^1(C^n; \ell^\infty_n)$, which is
$$
\int_{C^n} \max_{i=1,\dots, n} |\tilde K_i(z, {\bf 1}) d\mu(z)  = \int_{C^n} \max_{i=1,\dots, n} |K_i(z, {\bf 1}) d\mu(z) = \sum_{z\in C^n} \max_{i=1,\dots, n} |m_{i, z}|\,.
$$

Matrix elements $m_{i,z}$ were computed in Section \ref{T}, see \eqref{miz}.  So it is easy to calculate the latter quantity.

Formula  \eqref{miz} gives us the following: 
\begin{itemize}
\item if $d(z)=\dist(z, {\bf 1}) =0$, then  for all $i$ we have $m_{i, z}(\rho) = 2^{-n} (1+\rho)^{n-1}$,
\item if $d(z)=\dist(z, {\bf 1}) =1$, then  $\max_{i=1,\dots, n} |m_{i, z}| = 2^{-n}\int_0^1(1+\rho)^{n-1}d\rho$ again, 
\item if $d(z)=\dist(z, {\bf 1}) =2$, then  $\max_{i=1,\dots, n} |m_{i, z}| = 2^{-n}\int_0^1(1+\rho)^{n-2}(1-\rho)d\rho$, \dots, 
\item if $d(z)=\dist(z, {\bf 1}) =n$, then  $\max_{i=1,\dots, n} |m_{i, z}| = 2^{-n}\int_0^1(1-\rho)^{n-1}d\rho$.
\end{itemize}

Hence
\begin{align*}
&\sum_{z\in C^n} \max_{i=1,\dots, n} |m_{i, z}| = 2^{-n} \int_0^1 [(1+\rho)^{n-1} + n (1+\rho)^{n-1} +\frac{n(n-1)}{2} (1+\rho)^{n-2}(1-\rho)+
\\
& \dots + (1-\rho)^{n-1}]\, d\rho = 2^{-n}\int_0^1[n (1+\rho)^{n-1} +\frac{n(n-1)}{2} (1+\rho)^{n-2}(1-\rho)+ \dots + (1-\rho)^{n-1}]\frac{1-\rho}{1-\rho}\, d\rho +
\\
&+  2^{-n} \int_0^1 (1+\rho)^{n-1} d\rho =
\\
& 2^{-n}\int_0^1\frac{[(1+\rho)^n +n (1+\rho)^{n-1}(1-\rho) +\frac{n(n-1)}{2} (1+\rho)^{n-2}(1-\rho)^2+ \dots + (1-\rho)^{n} -(1+\rho)^n]}{1-\rho}\, d\rho +
\\
&+  2^{-n} \int_0^1 (1+\rho)^{n-1} d\rho =
\\
& 2^{-n}\int_0^1 \frac{ (1+\rho + 1-\rho)^n - (1+\rho)^n}{1-\rho}\, d\rho + 2^{-n} \int_0^1 (1+\rho)^{n-1} d\rho  = \int_0^1 \frac{1- (\frac{1+\rho}{2})^n}{1-\rho} \, d\rho + \frac{2^{-n}}{n}( 2^n -1)=
\\
&2 \int_{1/2}^1 \frac{1-x^n}{2(1-x)} \,dx +O\big(\frac1n\big)= \int_{1/2}^1 (1+x+x^2+\dots x^{n-1})\,dx + O\big(\frac1n\big) = 1+ \frac12 + \frac13+\dots+\frac1n + O(1)=
\\
& \log n +O(1)\,.
\end{align*}

We see that the operator norm $\sup_{f: \|f||_{L^1(\{-1, 1\}^n} \le 1}\|\int_0^\infty \nabla P_t f dt\|_{L^1(\{-1, 1\}^n; \ell^\infty_n)}$  grows logarithmically with $n$.  So $Curl$ space  plays major part for  the calculation of $C_\infty$ since we know from \cite{FP}, \cite{Wagner} that $C_\infty$ is bounded independent of $n$.
One cannot forget about $Curl$ space in this problem.

\section{Some formulas for $C_{dual}$}
\label{formulas}

As we know
   \begin{equation}
   \label{Cd}
C_{dual} =\lim_{n\to\infty}\sup_{\|v\|_{\ell_n^2} \le 1}  \bE_z\Big| \sum_{i=1}^n v_i\int_0^1\frac{z_i-\rho}{1-\rho^2} (1+\rho)^{n-d({\bf 1}, z)} (1-\rho)^{d({\bf 1}, z)}\, d\rho\Big|
\end{equation}

\bigskip

Let $B(a, b)= \int_0^1t^{a-1} (1-t)^{b-1}$ denote the beta function.  For $k=d({\bf 1}, z)$  integrals 
 $$
 2^{-n}\int_0^1\frac{z_i-\rho}{1-\rho^2} (1+\rho)^{n-d({\bf 1}, z)} (1-\rho)^{d({\bf 1}, z)}\, d\rho =\begin{cases} 
 \int_0^{1/2} t^k (1-t)^{n-k-1} dt ,\quad z_i=1
 \\
- \int_0^{1/2} t^{k-1} (1-t)^{n-k} dt ,\quad z_i=-1
  \end{cases}
 $$
  are certain combinatorial quantities called ``incomplete beta functions"  
 $$
 =
\begin{cases}
 B_{1/2}(k+1, n-k) = F_{1/2}(n-k-1, n) \cdot B(k+1, n-k), \quad \text{if}\,\, z_i=1,
\\
 B_{1/2}( k, n-k+1) = -F_{1/2}(n-k, n) \cdot B(k, n-k+1), \quad \text{if}\,\, z_i=-1\,.
 \end{cases}
 $$
 This can be written down as
 $$
 =
\begin{cases}
 B_{1/2}(k+1, n-k) = F_{1/2}(n-k-1, n) \frac1{(n-k)\binom{n}{k}}, \quad \text{if}\,\, z_i=1,
\\
 B_{1/2}( k, n-k+1) = -F_{1/2}(n-k, n)  \frac1{k\binom{n}{ k}}, \quad \text{if}\,\, z_i=-1\,.
 \end{cases}
 $$

  \medskip
 
 Here $F_{1/2}(n-k-1, n)$ is the probability of at least $k+1$ successes in Bernoulli scheme  with $n$ trials:
 $$
 F_{1/2}(n-k-1, n) =2^{-n} \sum_{r= k+1}^n\binom{n}{r}\,.
 $$
 
 Thus, we can rewrite the main part of  formula \eqref{Cd} for $C_{dual}$ as follows
 \begin{align*}
 & \bE_z\Big| \sum_{i=1}^n v_i\int_0^1\frac{z_i-\rho}{1-\rho^2} (1+\rho)^{n-d({\bf 1}, z)} (1-\rho)^{d({\bf 1}, z)}\, d\rho\Big| =
 \\
& = \frac1{2^n} \sum_{k=0}^n \binom{n}{k} \bE\Big[\Big| \sum_{i=1}^n v_i \frac1{\binom{n}{k}}\frac{2z_i}{1+(1-2k/n)z_i}  \sum_{r=k+1-1_{z_i=-1}}^n \binom{n}{r}  \Big|\Big| d({\bf 1}, z)=k \Big]
\\
&  = \frac1{2^n} \sum_{k=0}^n  \bE\Big[\Big| \sum_{i=1}^n v_i \frac{2z_i}{1+(1-2k/n)z_i}  \sum_{r=k+1-1_{z_i=-1}}^n \binom{n}{r}  \Big|\Big| d({\bf 1}, z)=k \Big]
\\
& = \frac1{n2^{n-1}} \sum_{k=0}^{n-1}  \sum_{r=k+1}^n \binom{n}{r}  \bE\Big[\Big| \sum_{i=1}^n v_i  \frac{z_i}{1+(1-2k/n)z_i}  \Big|\Big| d({\bf 1}, z)=k \Big]
\\
& + \frac1{2^n} \sum_{k=1}^n \frac1k\binom{n}{k}\bE\Big[\Big|  \sum_{i:\, z_i=-1} v_i  \Big|\Big| d({\bf 1}, z)=k \Big] 
 \end{align*}
 But $\|v\|_2\le 1$, so $\bE\Big[\Big|  \sum_{i:\, z_i=-1} v_i  \Big|\Big| d({\bf 1}, z)=k \Big]  \le \sqrt{k}$. Therefore, the last sum is
 $$
 \le  \frac1{2^n} \sum_{k=1}^n \frac1{\sqrt{k}}\binom{n}{k} \le \sqrt{ \frac1{2^n} \sum_{k=1}^n \frac1{k}\binom{n}{k}} \le \frac{2}{\sqrt{n+1}} \to 0.
 $$
Hence, we can rewrite
\begin{align*}
 & \bE_z\Big| \sum_{i=1}^{n-1} v_i\int_0^1\frac{z_i-\rho}{1-\rho^2} (1+\rho)^{n-d({\bf 1}, z)} (1-\rho)^{d({\bf 1}, z)}\, d\rho\Big| =
 \\
 & \frac1{n2^{n-1}} \sum_{k=0}^n  \sum_{r=k+1}^n \binom{n}{r}  \bE\Big[\Big| \sum_{i=1}^n v_i  \frac{z_i}{1+(1-2k/n)z_i}  \Big|\Big| d({\bf 1}, z)=k \Big] + o(1)\,.
 \end{align*}
 
 Here is the first formula for $C_{dual}$:
 \begin{equation}
   \label{Cdfirst}
C_{dual} =\lim_{n\to\infty}\sup_{\|v\|_{\ell_n^2} \le 1}  \frac1{n2^{n-1}} \sum_{k=0}^{n-1}  \sum_{r=k+1}^n \binom{n}{r}  \bE\Big[\Big| \sum_{i=1}^n v_i  \frac{z_i}{1+(1-2k/n)z_i}  \Big|\Big| d({\bf 1}, z)=k \Big] \,.
\end{equation}

\bigskip

We can simplify further this formula. For that, consider
$$
\Phi_n(k) :=  \frac1{2^n}\sum_{r=k+1}^n \binom{n}{r}  =\bP[\text{Binom}(n, \frac12)>k]\,.
$$
This implies
$$
\Phi_n(\frac{\sqrt{n}}{2} +\frac{x\sqrt{n}}{2} )\to \bP[N(0,1) >x]\,.
$$

\begin{theorem}
\label{Cdsecond-th}
 \begin{equation}
   \label{Cdsecond}
C_{dual} =\lim_{n\to\infty}\sup_{\|v\|_{\ell_n^2} \le 1}  \frac{2}{n}\sum_{k=0}^{n/2}    \bE\Big[\Big| \sum_{i=1}^n v_i  \frac{z_i}{1+(1-2k/n)z_i}  \Big|\Big| d({\bf 1}, z)=k \Big] \,.
\end{equation}
\end{theorem}
\begin{proof}
Define
$$
X_i^k := \frac{z_i}{1+(1-2k/n)z_i} \,.
$$
By Cauchy--Schwarz we get
\begin{equation}
\label{CS}
\bE [\sum_{i=1}^n v_i  X_i^k | d({\bf 1}, z)=k ] \le \sum_{i, j=1}^nv_iv_j \bE[ X_j^k X_i^k |d({\bf 1}, z)=k ]\,.
\end{equation}
It is easy to see that
\begin{align*}
&\bE [X_i^k | d({\bf 1}, z)=k ]=0,
\\
& \bE [(X_i^k)^2 | d({\bf 1}, z)=k ]=\frac{n^2}{4k(n-k)},
\\
&\bE [X_i^k X_j^k | d({\bf 1}, z)=k ]=-\frac1{n-1}\frac{n^2}{4k(n-k)},\quad i\neq j.
\end{align*}
Hence, we know the Gramm matrix $G:=\{\bE[ X_j^k X_i^k |d({\bf 1}, z)=k ]\}$. It is self adjoint rank one perturbation of the diagonal matrix $\text{diag} (\frac{n}{n-1}\frac{n^2}{4k(n-k)})$. So it is easy to calculate the norm of $G$. It is $\sqrt{\frac{n}{n-1}\frac{n^2}{4k(n-k)}}$.

Hence, by \eqref{CS}
\begin{equation}
\label{CS1}
\bE [\sum_{i=1}^n v_i  X_i^k | d({\bf 1}, z)=k ] \le \sqrt{\frac{n}{n-1}\frac{n^2}{4k(n-k)}}  \,.
\end{equation}

Let $b_n= \sqrt{n} \log n$. By the same spectral estimate for a cut-off matrix and by the fact that $\Phi_n(k) \le 1$, we will get
\begin{equation}
\label{CS2}
\sum_{k=n/2 -b_n}^{n/2+b_n} \Phi_n(k)\bE [\sum_{i=1}^n v_i  X_i^k | d({\bf 1}, z)=k ] \le \sum_{k=n/2 -b_n}^{n/2+b_n}  \sqrt{\frac{n}{n-1}\frac{n^2}{4k(n-k)}} =O(b_n) \,.
\end{equation}
Thus,
\begin{equation}
\label{CS3}
 \frac1{n2^{n-1}} \sum_{k=n/2 -b_n}^{n/2+b_n}   \sum_{r=k+1}^n \binom{n}{r}  \bE\Big[\Big| \sum_{i=1}^n v_i  \frac{z_i}{1+(1-2k/n)z_i}  \Big|\Big| d({\bf 1}, z)=k \Big]  = O\big(\frac{\log n}{\sqrt{n}}\big) \,.
\end{equation}
Using the spectral estimate for another cut-off matrix and the fact that $\sqrt{\frac{n}{n-1}\frac{n^2}{4k(n-k)}} = O(\sqrt{n})$, we get
\begin{equation}
\label{CS4}
 \bE [\sum_{i=1}^n v_i  X_i^k | d({\bf 1}, z)=k ] \le \sqrt{\frac{n}{n-1}\frac{n^2}{4k(n-k)}} =O(\sqrt{n}) \,.
\end{equation}
Hence,
\begin{align}
\label{CS5}
& \sum_{k=n/2 + b_n}^{n-1}  \Phi_n(k) \bE\Big[\Big| \sum_{i=1}^n v_i  \frac{z_i}{1+(1-2k/n)z_i}  \Big|\Big| d({\bf 1}, z)=k \Big]   = O\Big(\sqrt{n}\sum_{k=n/2 + b_n}^{n-1} \Phi_n(k)\Big)  \notag
\\
&= O(n^{3/2} \Phi_n(n/2+b_n) ) = O(n^{3/2} e^{-\log^2 n})\,.
\end{align}
\begin{equation}
\label{CS6}
 \frac1{n2^{n-1}} \sum_{k=n/2 + b_n}^{n-1}   \sum_{r=k+1}^n \binom{n}{r}  \bE\Big[\Big| \sum_{i=1}^n v_i  \frac{z_i}{1+(1-2k/n)z_i}  \Big|\Big| d({\bf 1}, z)=k \Big]  = O\big(n^{1/2} e^{-\log^2 n}\big) \,.
\end{equation}
So when we come to the limit with $n\to \infty$ in \eqref{Cdfirst}, we are left with
\begin{align*}
&C_{dual}=\lim_{n\to\infty}\sup_{\|v\|_{\ell_n^2} \le 1}  \frac1{n2^{n-1}} \sum_{k=n/2 -b_n}^{n-1}  \sum_{r=k+1}^n \binom{n}{r}  \bE\Big[\Big| \sum_{i=1}^n v_i  \frac{z_i}{1+(1-2k/n)z_i}  \Big|\Big| d({\bf 1}, z)=k \Big] =
\\
&= \lim_{n\to\infty}\sup_{\|v\|_{\ell_n^2} \le 1}  \frac{2}{n} \sum_{k=n/2 -b_n}^{n-1} \Phi_n(k)\bE\Big[\Big| \sum_{i=1}^n v_i  \frac{z_i}{1+(1-2k/n)z_i}  \Big|\Big| d({\bf 1}, z)=k \Big] =
\\
& =  \lim_{n\to\infty}\sup_{\|v\|_{\ell_n^2} \le 1}  \frac{2}{n} \sum_{k=n/2 -b_n}^{n-1} \bE\Big[\Big| \sum_{i=1}^n v_i  \frac{z_i}{1+(1-2k/n)z_i}  \Big|\Big| d({\bf 1}, z)=k \Big],
\end{align*}
the last equality being true because in this range of $k$ we have
$$
1-O(e^{-\log^2 n}) \le \Phi_n(k) \le 1\,.
$$

Now we can add back 
$$
 \frac{2}{n} \sum_{k=n/2 -b_n}^{n/2+b_n} \bE\Big[\Big| \sum_{i=1}^n v_i  \frac{z_i}{1+(1-2k/n)z_i}  \Big|\Big| d({\bf 1}, z)=k \Big]
 $$
 and
 $$ \frac{2}{n} \sum_{k=n/2 +b_n}^{n-1} \bE\Big[\Big| \sum_{i=1}^n v_i  \frac{z_i}{1+(1-2k/n)z_i}  \Big|\Big| d({\bf 1}, z)=k \Big]
 $$ 
 as we already saw that they are $ O\big(\frac{\log n}{\sqrt{n}}\big) $ and $O\big( e^{-\log^2 n}\big)$ correspondingly. Theorem is completely proved.
\end{proof}

\begin{remark}
From \eqref{CS1} and \eqref{Cdfirst} one immediately deduces the estimate
\begin{equation}
\label{LP}
C_{dual} \le \lim_{n\to\infty} \sum_{k=1}^{n-1} \Phi_n(k) \frac1{\sqrt{k(n-k)} }=\frac{\pi}{2}\,.
\end{equation}
Of course Theorem \ref{Cdsecond-th} also immediately gives this estimate.
This is one of the numerous re-proofs of Lust-Piquard--Ben Efraim estimate from \cite{BELP}. As we can see the reason for their estimate lies in Cauchy--Schwarz estimate and spectral calculation for the Gramm matrix.
\end{remark} 

\bigskip

\begin{remark}
What we did in this paper, was that we estimated by funny Khintchine inequality another \textup(and formally bigger than $C_{dual}$\textup) quantity
\begin{equation}
\label{Cd1}
\lim_{n\to\infty}\int_0^1\sup_{ \|v\|_{\ell_n^2} \le 1}  \bE_z\Big| \sum_{i=1}^n v_i\frac{z_i-\rho}{1-\rho^2} \Big|(1+\rho)^{n-d({\bf 1}, z)} (1-\rho)^{d({\bf 1}, z)}\, d\rho <\frac{\pi}{2}\,.
\end{equation}
\end{remark}



\end{document}